\documentclass[a4paper, 11pt]{amsart}

\pdfoutput=1

\usepackage{amsfonts}
\usepackage{amssymb}
\usepackage{microtype}
\usepackage{hyperref} 

\newtheorem{theorem}{Theorem}[section]
\newtheorem{proposition}[theorem]{Proposition}
\newtheorem{lemma}[theorem]{Lemma}
\newtheorem{corollary}[theorem]{Corollary}
\theoremstyle{definition}
\newtheorem{definition}{Definition}[section]
\newtheorem{assumption}{Assumption}
\newtheorem{remark}[theorem]{Remark}

\numberwithin{equation}{section}

\newcommand{\naturals}{\mathbb{N}}

\newcommand{\reals}{\mathbb{R}}

\newcommand{\hcal}{\mathcal{H}}
\newcommand{\pcal}{\mathcal{P}}

\newcommand{\diff}{\mathop{}\!d}

\newcommand{\norm}[1]{\|#1\|}

\newcommand{\lipnorm}[1]{\norm{#1}_{\text{Lip}}}

\newcommand{\charfun}{\mathbf{1}}

\DeclareMathOperator{\inj}{inj}
\DeclareMathOperator{\diam}{diam}
\DeclareMathOperator{\supp}{supp}
\DeclareMathOperator{\dist}{dist}
\DeclareMathOperator{\vol}{Vol}

\newcommand{\leqs}{\lesssim}
\newcommand{\geqs}{\gtrsim}

\title[Borel--Cantelli Lemma for {Axiom A} Diffeomorphisms]
{
  A Recurrence-type Strong Borel--Cantelli Lemma for {Axiom A}
  Diffeomorphisms
}

\author{Alejandro Rodriguez Sponheimer}

\date{February 7, 2025}

\address{Centre for Mathematical Sciences,
Lund University, Box~118, 221~00 Lund, Sweden}

\email{alejandro.rodriguez\_sponheimer@math.lth.se}



\begin{document}

\begin{abstract}
  Let $(X,\mu,T,d)$ be a metric measure-preserving dynamical system
  such that $3$-fold correlations decay exponentially for Lipschitz
  continuous observables. Given a sequence $(M_k)$ that converges to
  $0$ slowly enough, we obtain a strong dynamical Borel--Cantelli
  result for recurrence, i.e., for $\mu$-a.e.\ $x\in X$
  \[
    \lim_{n \to \infty}\frac{\sum_{k=1}^{n} \charfun_{B_k(x)}(T^{k}x)}
    {\sum_{k=1}^{n} \mu(B_k(x))} = 1,
  \]
  where $\mu(B_k(x)) = M_k$. In particular, we show that this result
  holds for Axiom~A diffeomorphisms and equilibrium states under
  certain assumptions.
\end{abstract}

\maketitle

\section{Introduction}
The Poincar\'e recurrence theorem is a fundamental result in dynamical
systems. It states that if $(X, \mu, T)$ is a separable
measure-preserving dynamical system with a Borel probability measure
$\mu$, then $\mu$-a.e.\ point $x \in X$ is recurrent, i.e., there
exists an increasing sequence $n_i$ such that $T^{n_i}x \to x$.
Equipping the system with a metric $d$ such that $(X,d)$ is separable
and $d$-open sets are measurable, we may restate the result as
$\liminf_{n \to \infty} d(T^{n}x,x) = 0$.
A natural question to ask is if anything can be said about the rate of 
convergence. In \cite{boshernitzan1993Quantitative}, Boshernitzan
proved that if the Hausdorff measure $\hcal_\alpha$ is $\sigma$-finite
on $X$ for some 
$\alpha > 0$, then for $\mu$-a.e.\ $x \in X$
\[
  \liminf_{n \to \infty} n^{1/\alpha} d(x,T^{n}x) < \infty.
\]
Moreover, if $H_{\alpha}(X) = 0$, then for $\mu$-a.e.\ $x\in X$ 
\[
  \liminf_{n \to \infty} n^{1/\alpha} d(x,T^{n}x) = 0.
\]

We may state Boshernitzan's result in a different way: for
$\mu$-a.e.\ $x\in X$ there exists a constant $c > 0$ such that
\[
  \sum_{k=1}^{\infty} \charfun_{B(x,ck^{-1/\alpha})}(T^{k}x) 
  = \infty,
\]
where $B(x,r)$ denotes the open ball with centre $x$ and radius $r$.
Such a sum resembles those that appear in dynamical Borel--Cantelli
lemmas
for shrinking targets, $B_k = B(y_k,r_k)$, where the centres $y_k$ do
not depend on $x$. A dynamical Borel--Cantelli lemma is a zero-one law
that gives conditions on the dynamical system and a sequence of sets
$A_k$ such that
$\sum_{k=1}^{\infty} \charfun_{A_k}(T^{k}x)$ converges or diverges for
$\mu$-a.e.\ $x$, depending on the convergence or divergence of
$\sum_{k=1}^{\infty} \mu(A_k)$.
In some cases it is possible to prove the stronger result,
namely that if $\sum_{k=1}^{\infty} \mu(A_k) = \infty$, then
$\sum_{k=1}^{\infty} \charfun_{A_k}(T^{k}x) 
\sim \sum_{k=1}^{\infty} \mu(A_k)$.
Such results are called strong dynamical Borel--Cantelli lemmas.

Although the first dynamical Borel--Cantelli lemma for shrinking
targets was proved in \textup{1967} by Philipp
\cite{philipp1967Metrical}, recurrence versions have only recently been
studied. The added difficulty arises because of the dependence on $x$
that the sets $B_{k}(x)$ have.
Persson proved in \cite{persson2023Strong} that one can obtain a
recurrence version of a strong dynamical Borel--Cantelli lemma for a
class of mixing dynamical systems on the unit interval.
Other recent recurrence results include
Hussain, Li, Simmons and Wang~\cite{hussain2022Dynamical},
who obtained a zero-one law under the assumptions of mixing conformal
maps and Ahlfors regular measures. Under similar assumptions,
Kleinbock and Zheng~\cite{kleinbock2023Dynamical}
prove a zero-one law under Lipschitz twists, which combines recurrence
and shrinking targets results.
Both these papers put strong assumptions on the underlying measure of
the dynamical system, assuming, for example, Ahlfors regularity.
For other recent improvements of Boshernitzan's result, see
\S\ref{final:background}.

The main result of the paper, Theorem~\ref{thm:general},
is a strong dynamical Borel--Cantelli result for recurrence for
dynamical systems satisfying some general conditions. The theorem
extends the main theorem of Persson \cite{persson2023Strong}, which is
stated in Theorem~\ref{thm:persson}, concerning a class of mixing
dynamical systems on the unit interval to a more general class, that
includes, for example, some mixing systems on compact smooth manifolds.
Since we consider shrinking balls $B_k(x)$ for which the radii converge
to $0$ at various rates, Theorem~\ref{thm:general} is an improvement of
Boshernitzan's result for the systems considered.
For now we state the following result, which is the main application of
Theorem~\ref{thm:general}.
Equip a smooth manifold with the induced metric.

\begin{theorem}\label{thm:main} 
  Consider a compact smooth $N$-dimensional manifold $M$. Suppose
  $f \colon M \to M$ is an Axiom~A diffeomorphism that is topologically
  mixing on a basic set and $\mu$ is an equilibrium state corresponding
  to a H\"older continuous potential on the basic set. Suppose that
  there exist $c_1>0$ and $s > N - 1$ satisfying
  \[
    \mu(B(x,r)) \leq c_1r^{s}
  \]
  for all $x \in M$ and $r \geq 0$. Assume further that $(M_n)$ is a
  sequence converging to $0$ satisfying
  \[
    M_n \geq \frac{(\log n)^{4+\varepsilon}}{n}
  \]
  for some $\varepsilon > 0$ and
  \[
    \lim_{\alpha \to 1^{+}}\limsup_{n\to \infty}
      \frac{M_n}{M_{\lfloor \alpha n\rfloor}}
    = 1.
  \]
  Define $B_n(x)$ to be the ball at $x$ with $\mu(B_n(x)) = M_n$. Then
  \[
    \lim_{n \to \infty}\frac{\sum^{n}_{k=1} \charfun_{B_k(x)}(f^{k}x)}
    {\sum^{n}_{k=1} \mu(B_k(x))} = 1
  \]
  for $\mu$-a.e.\ $x \in M$.
\end{theorem}

The above theorem applies to hyperbolic toral automorphisms with the
Lebesgue measure as the equilibrium measure, which corresponds to the
$0$ potential. It is clear that the Lebesgue measure satisfies the
assumption. By perturbing the potential by a H\"older continuous
function with small norm, we obtain a different equilibrium state for
which the above assumption on the measure holds as well. Equilibrium
states that are absolutely continuous with respect to the Lebesgue
measure also satisfy the assumption.

For shrinking targets (non-recurrence),
Chernov and Kleinbock~\cite[Theorem~\textup{2.4}]{chernov2001Dynamical}
obtain a strong Borel--Cantelli lemma when $T$ is an Anosov
diffeomorphism, $\mu$ is an equilibrium state given by a H\"older
continuous potential, and the targets are \emph{eventually quasi-round}
rectangles.

\subsection*{Paper Structure}

In Section~\ref{sec:results}, we state the main result of the paper,
namely Theorem~\ref{thm:general}, and give examples of systems
satisfying the assumptions. Applications to return times and pointwise
dimension are also given.
In Section~\ref{sec:lemmas}, we prove a series of lemmas establishing
properties of the measure $\mu$ and functions 
$r_n \colon X \to [0,\infty)$ defined as the radius of the balls
$B_n(x)$. 
In Section~\ref{sec:props}, we obtain estimates for the measure and
correlations of the sets $E_n = \{x\in X : T^{n}x \in B_n(x)\}$.
This is done in Propositions~\ref{prop:main1} and \ref{prop:main2}
using indicator functions and decay of correlations.
Finally, in Section~\ref{sec:thm:proof}, we use the previous
propositions together with Theorem~\ref{thm:persson} to prove
Theorem~\ref{thm:general}. Theorem~\ref{thm:main} then follows once we
show that the assumptions of Theorem~\ref{thm:general} are satisfied.
We conclude with final remarks (Section~\ref{sec:finalremarks}).

\section{Results}\label{sec:results}
\subsection{Setting and Notation}
We say that $T$ preserves the measure $\mu$ if $\mu(T^{-1}A) = \mu(A)$
for all $\mu$-measurable sets $A$. From now on, $(X,\mu,T,d)$ will
denote a metric measure-preserving system (m.m.p.s.) for which $(X,d)$
is compact, $\mu$ is a Borel probability measure, and $T$ is a
measurable transformation. Given a sequence $(M_n)$ in $[0,1]$ and
ignoring issues for now, define the open ball $B_n(x)$ around $x$ by
$\mu(B_n(x)) = M_n$ and define the functions
$r_n \colon X \to [0,\infty)$ by $B(x,r_n(x)) = B_n(x)$.
For each $n \in \naturals$ define
\[
  E_n = \{x \in X : T^{n}x \in B(x,r_n(x))\}.
\]
For a set $A$ we denote the diameter, cardinality and closure of $A$
by $\diam A$, $|A|$ and $\overline{A}$ respectively. By the
$\delta$-neighbourhood of $A$ we mean the set
$A(\delta) = \{x \in X : \dist(x,A) < \delta\}$.
For $a, b \in \reals$ we write $a \leqs b$ to mean that there exists a
constant $c > 0$ depending only on $(X,\mu,T,d)$ such that $a \leq cb$.
For $V\subset X$ and $\varepsilon > 0$ let $P_{\varepsilon}(V)$ denote
the packing number of $V$ by balls $B$ of radius $\varepsilon$. This is
the maximum number of pairwise disjoint balls of radius $\varepsilon$
with centres in $V$. Since $X$ is compact, $P_{\varepsilon}(V)$ is
finite for all $V\subset X$ and $\varepsilon > 0$.
For a compact smooth manifold $M$ let $d$ denote the
induced metric and let $\vol$ denote the induced volume measure. The
injectivity radius of $M$, 
which is the largest radius for which the exponential map at every
point is a diffeomorphism,
is denoted by $\inj_M$.
Since $M$ is compact, $\inj_M > 0$ (see for instance Chavel
\cite{chavel2006Riemannian}).

Note that the open balls $B_n(x)$ and radii $r_n(x)$ may not
exist for some $n \in \naturals$ and $x \in X$. Sufficient conditions
on $\mu$, $x$ and $(M_n)$ for which $B_n(x)$ and $r_n(x)$ exist are
given in Lemma~\ref{lemma:r:existence}.
Notice that for a zero-one law one would prove that 
$\limsup_{n\to \infty}E_n$ has either zero or full measure.
Furthermore, in the shrinking target case, when one considers fixed
targets $B_n = B(y_n,r_n)$, the corresponding sets are
$\tilde{E}_n = \{ x \in X : T^{n}x \in B_n\} = T^{-n} B_n$
and we can use the invariance of the measure $\mu$ to conclude that
$\mu(\tilde{E}_n) = M_n$. However, in the recurrence case, when the
targets depend on $x$, we instead settle for estimates on the measure
and correlations of the sets $E_n$. We do so in
Proposition~\ref{prop:main3} using decay of multiple correlations as
our main tool. We state the definition of decay of multiple
correlations for H\"older continuous observables as it is the form most
commonly found in literature; however, for our purposes we only require
decay of correlations for Lipschitz continuous observables.

\begin{definition}[$r$-fold Decay of Correlations]
  For an m.m.p.s\ $(X,\mu,T,d)$, $r \in \naturals$ and
  $\theta \in (0,1]$, we say that \emph{$r$-fold correlations decay
  exponentially for $\theta$-H\"older continuous observables} if there
  exists constants $c_1 >0$ and $\tau \in (0,1)$ such that for all
  $\theta$-H\"older continuous functions
  $\varphi_k : X \to \reals$, $k=0,\dotsc,r-1$, and integers
  $0 = n_0 < n_1 < \dotsc < n_{r-1}$
  \[
    \Bigl| \int \prod_{k=0}^{r-1} \varphi_k \circ T^{n_k} \diff\mu -
    \prod_{k=0}^{r-1} \int \varphi_k \diff\mu \Bigr|
    \leq c_1 e^{-\tau n} \prod_{k=0}^{r-1} \norm{\varphi_k}_{\theta},
  \]
  where $n = \min \{n_{i+1} - n_{i}\}$, and
  \[
    \norm{\varphi}_{\theta} 
    = \sup_{x\neq y} \frac{|\varphi(x) - \varphi(y)|}{d(x,y)^{\theta}} 
    + \norm{\varphi}_{\infty}.
  \]
\end{definition}

It is well known that if $M$ is a compact smooth manifold, $f$ an
Axiom~A diffeomorphism that is topologically mixing when restricted to
a basic set $\Omega$, and $\varphi \colon \Omega \to \reals$
H\"older continuous, then there exists a unique equilibrium state
$\mu_\varphi$ for $\varphi$ such that $2$-fold correlations decay
exponentially for $\theta$-H\"older continuous observables for all
$\theta \in (0,1]$ (see Bowen \cite{bowen2008Equilibrium}).
In the same setting, Kotani and Sunada essentially prove in
\cite[Proposition~3.1]{kotani2001Pressure} that, in fact, $r$-fold
correlations decay exponentially for $\theta$-H\"older continuous
observables for all $r \geq 1$ and $\theta \in (0,1]$.
A similar result is proven in \cite[Theorem~7.41]{chernov2006Chaotic}
for billiards with bounded horizon and no corners and for
\emph{dynamically H\"older continuous functions} -- a class of
functions that includes H\"older continuous functions. This result can
be extended to billiards with no corners and unbounded horizon, and
billiards with some corners and bounded horizon.
Dolgopyat \cite{dolgopyat2004Limit} proves that multiple correlations
decay exponentially for partially hyperbolic systems that are `strongly
u-transitive with exponential rate.' Dolgopyat lists systems in
Section~6 of his paper for which this property holds:
some time one maps of Anosov flows, quasi-hyperbolic toral
automorphisms, some translations on homogeneous spaces, and mostly
contracting diffeomorphisms on $3$-dimensional manifolds.
A similar property referred to as `Property $(\pcal_t)$' for real
$t \geq 1$ is used by P\`ene in \cite{pene2004Multiple}. It is stated
that when $t > 1$, property $(\pcal_t)$ holds for dynamical systems to
which one can apply Young's method in \cite{young1998Statistical}.

We state some assumptions on the system $(X,\mu,T,d)$ and the sequence
$(M_n)$ that are used to obtain our result.

\begin{assumption}\label{assumption:seq}
  There exists $\varepsilon > 0$ such that
  \[
    M_n \geq \frac{(\log n)^{4 + \varepsilon}}{n}
  \]
  and
  \[
    \lim_{\alpha \to 1^{+}}\limsup_{n \to \infty}
    \frac{M_n}{M_{\lfloor \alpha n \rfloor}} = 1.
  \]
\end{assumption}

\begin{assumption}\label{assumption:ball}
  There exists $s>0$ such that for all $x \in X$ and $r>0$
  \[
    \mu(B(x,r)) \leqs r^{s}.
  \]
\end{assumption}

\begin{assumption}\label{assumption:annulus}
  There exists positive constants $\rho_0$ and $\alpha_0$ such that
  for all $x \in X$ and $0 < \varepsilon < \rho \leq \rho_0$
  \[
    \mu \{y : \rho \leq d(x,y) < \rho + \varepsilon\} 
    \leqs \varepsilon^{\alpha_0}.
  \]
\end{assumption}

\begin{assumption}\label{assumption:packing}
  There exists positive constants $K$ and $\varepsilon_0$ such that
  for any $\varepsilon \in (0,\varepsilon_0)$ the packing number of
  $X$ satisfies
  \[
    P_{\varepsilon}(X) \leqs \varepsilon^{-K}.
  \]
\end{assumption}

Assumption~\ref{assumption:seq} is needed for technical reasons that
are explicit in \cite[Lemma~\textup{2}]{persson2023Strong}.
Assumption~\ref{assumption:ball} is a standard assumption.
Assumption~\ref{assumption:annulus} is more restrictive; however, in
some cases it can be deduced from Assumption~\ref{assumption:ball}. 
Assumption~\ref{assumption:packing} holds for very general spaces,
e.g., all compact smooth manifolds.
As we will see in the proof of Theorem~\ref{thm:main},
Assumptions~\ref{assumption:annulus} and \ref{assumption:packing} hold
when $X$ is a compact smooth $N$-dimensional manifold and $s > N - 1$.
Assumption~\ref{assumption:annulus} is also used in
\cite{haydn2013Note}, stated as Assumption~B, in which the authors
outline spaces that satisfy the assumption. For instance, the
assumption is satisfied by
dispersing billiard systems,
compact group extensions of Anosov systems,
a class of Lozi maps and
one-dimensional non-uniformly expanding interval maps with invariant
probability measure $d\mu = h\mathop{}\!d\lambda$, where $\lambda$ is
the Lebesgue measure and $h \in L^{1+\delta}(\lambda)$ for some
$\delta > 0$.

Assumption~\ref{assumption:packing} is used in the following
construction of a partition.
Let $\varepsilon < \varepsilon_0$ and let
$\{B(x_k,\varepsilon)\}_{k=1}^{L}$ be a maximal $\varepsilon$-packing
of $X$, i.e., $L = P_{\varepsilon}(X)$. Then
$\{B(x_k,2\varepsilon)\}_{k=1}^{L}$ covers $X$. Indeed, for any
$x \in X$ we must have that $d(x,x_k) < 2\varepsilon$ for some $k$, for
we could otherwise add $B(x,\varepsilon)$ to the packing, contradicting
our maximality assumption. Now, let $A_1 = B(x_1,2\varepsilon)$ and
recursively define
\begin{equation}\label{eq:X:partition}
  A_k = B(x_k,2\varepsilon) \setminus 
  \bigcup_{i=1}^{k-1} B(x_i,2\varepsilon)
\end{equation}
for $k = 2,\dotsc,L$.
Then $\{A_k\}_{k=1}^{L}$ partitions $X$ and satisfies
$\diam A_k < 4\varepsilon$. By Assumption~\ref{assumption:packing},
$L \leqs \varepsilon^{-K}$.

As previously mentioned, Theorem~\ref{thm:general} extends Persson's
result in \cite{persson2023Strong} by relaxing the conditions on
$(X,\mu,T,d)$. In \cite{persson2023Strong}, it is assumed that
$X = [0,1]$ and that $T \colon [0,1] \to [0,1]$ preserves a measure
$\mu$ for which ($2$-fold) correlations decay exponentially for $L_1$
against $BV$ functions. This is satisfied, for example, when $T$ is
piecewise uniformly expanding map and $\mu$ is a Gibbs measure.
That correlations decay exponentially for $L_1$ against $BV$ functions
allows one to directly apply it to indicator functions.
As established, when $T$ is an Axiom~A diffeomorphism defined
on a manifold $X$, multiple correlations decay exponentially for
H\"older continuous functions. Thus, we may not directly apply the
result to indicator functions and must use an approximation argument,
complicating the proof. That $\mu$ is non-atomic in
\cite{persson2023Strong} translates in Theorem~\ref{thm:main} to $\mu$
satisfying Assumption~\ref{assumption:ball} with $s > \dim M - 1$.
Furthermore, Assumption~\ref{assumption:ball} for $\mu$ and
Assumption~\ref{assumption:seq} for $(M_n)$ remain unchanged. As it
will be of use later, we restate the main result of
\cite{persson2023Strong} by combining Proposition~1 and the main
theorem. We do so in order to have a more general statement suitable
for our use.

\begin{theorem}[(Persson \cite{persson2023Strong})]\label{thm:persson}
  Let $(X, \mu, T, d)$ be a m.m.p.s., $(M_n)$ a sequence in $[0,1]$
  satisfying Assumption~\ref{assumption:seq} and $\mu$ a measure
  satisfying Assumption~\ref{assumption:ball}. Define $B_n(x)$ to be
  the ball around $x$ such that $\mu(B_n(x)) = M_n$ and let
  $E_n = \{x \in X : T^{n}x \in B_n(x)\}$.
  Suppose that there exists $C,\eta > 0$ such that for all
  $n,m \in \naturals$
  \begin{equation}\label{eq:persson:measure}
    | \mu(E_n) - M_n | \leq C e^{-\eta n}
  \end{equation}
  and
  \begin{equation}\label{eq:persson:correlation}
    \mu(E_{n+m} \cap E_{n}) \leq \mu(E_{n+m})\mu(E_n)
    + C(e^{-\eta n} + e^{-\eta m}).
  \end{equation}
  Then
  \[
    \lim_{n \to \infty}\frac{\sum^{n}_{k=1} \charfun_{B_k(x)}(T^{k}x)}
    {\sum^{n}_{k=1} \mu(B_k(x))} = 1
  \]
  for $\mu$-a.e.\ $x \in X$.
\end{theorem}

In this paper, we establish inequalities \eqref{eq:persson:measure} and
\eqref{eq:persson:correlation} for systems satisfying multiple
decorrelation for Lipschitz continuous observables and conclude a
strong Borel--Cantelli lemma for recurrence.

\subsection{Main Result}
The main result of the paper is the following theorem.

\begin{theorem}\label{thm:general}
  Let $(X,\mu,T,d)$ be an m.m.p.s.\ for which $3$-fold correlations
  decay exponentially for Lipschitz continuous observables and such
  that Assumptions~\ref{assumption:ball}--\ref{assumption:packing}
  hold. Suppose further that $(M_n)$ is a sequence in $[0,1]$
  converging to $0$ and satisfying Assumption~\ref{assumption:seq}.
  Then
  \[
    \lim_{n \to \infty}\frac{\sum^{n}_{k=1} \charfun_{B_k(x)}(T^{k}x)}
    {\sum^{n}_{k=1} \mu(B_k(x))} = 1
  \]
  for $\mu$-a.e.\ $x \in X$.
\end{theorem}

\begin{remark}
  Although we consider open balls $B_k(x)$ for our results, we only
  require control over the measure of $B_k(x)$ and its
  $\delta$-neighbourhoods. Thus, it is possible to substitute the open
  balls for closed balls or other neighbourhoods of $x$, as long as one
  retains similar control of the sets as stated in
  Assumptions~\ref{assumption:ball} and \ref{assumption:annulus}.
  Furthermore, it is possible to relax the condition that $(M_n)$ 
  converges to $0$ as long as one can ensure that $r_n \leq \rho_0$ on
  $\supp\mu$ for large $n$.
\end{remark}

After proving that Assumptions~\ref{assumption:annulus} and
\ref{assumption:packing} hold when $X$ is a compact smooth manifold
and $\mu$ satisfies Assumption~\ref{assumption:ball} for $s > N - 1$,
we apply the above result to Axiom~A diffeomorphisms and obtain
Theorem~\ref{thm:main}.
In \S\ref{final:applications} we list some systems that satisfy
Assumptions~\ref{assumption:ball}--\ref{assumption:packing}, but it
remains to show that some of these systems satisfy $3$-fold decay of
correlations.

\subsection{Return Times Corollary}
Define the \emph{hitting time} of $x\in X$ into a set $B$ by
\[
  \tau_{B}(x) = \inf \{k > 0 : T^{k}x \in B\}.
\]
When $x \in B$ we call $\tau_{B}(x)$ the \emph{return time} of $x$
into $B$.
In \textup{2007}, Galatolo and Kim \cite{galatolo2007Dynamical}
showed that if a system satisfies a Strong Borel--Cantelli lemma for
shrinking targets with any centre, then for every $y\in X$ the hitting
times satisfy 
\[
  \lim_{r \to 0} \frac{\log \tau_{\overline{B}(y,r)}(x)}
  {-\log \mu(\overline{B}(y,r))} = 1
\]
for $\mu$-a.e.\ $x$.
In our case, the proof translates directly to return times (see the
corollary in \cite{persson2023Strong} for a simplified proof).
\begin{corollary}\label{corollary:return}
  Under the assumptions of Theorem \ref{thm:general},
  \[
    \lim_{r \to 0} \frac{\log \tau_{B(x,r)}(x)}{-\log \mu(B(x,r))} = 1
  \]
  for $\mu$-a.e.\ $x$.
\end{corollary}
The above corollary has an application to the pointwise dimension of
$\mu$. For each $x \in X$ define the \emph{lower} and
\emph{upper pointwise dimensions} of $\mu$ at $x$ by
\[
  \underline{d}_{\mu}(x)
    = \liminf_{r \to 0} \frac{\log \mu(B(x,r))}{\log r}
  \quad \text{and} \quad
  \overline{d}_{\mu}(x)
    = \limsup_{r \to 0} \frac{\log \mu(B(x,r))}{\log r},
\]
and the \emph{lower} and \emph{upper recurrence rates} of $x$ by
\[
  \underline{R}(x)
    = \liminf_{r \to 0}\frac{\log\tau_{B(x,r)}(x)}{-\log r}
  \quad \text{and} \quad
  \overline{R}(x)
    = \limsup_{r \to 0}\frac{\log\tau_{B(x,r)}(x)}{-\log r}.
\]
When $T$ is a Borel measurable transformation that preserves a Borel
probability measure $\mu$ and $X \subset \reals^{N}$ is any measurable
set, Barreira and Saussol
\cite[Theorem~\textup{1}]{barreira2001Hausdorff} proved that
\[
  \underline{R}(x) \leq \underline{d}_{\mu}(x)
  \quad \text{and} \quad
  \overline{R}(x) \leq \overline{d}_{\mu}(x)
\]
for $\mu$-a.e.\ $x\in X$.
Using Corollary~\ref{corollary:return}, we obtain that for 
$\mu$-a.e.\ $x\in X$
\begin{equation}\label{eq:returndim}
  \underline{R}(x) = \underline{d}_{\mu}(x)
  \quad \text{and} \quad
  \overline{R}(x) = \overline{d}_{\mu}(x),
\end{equation}
whenever the system satisfies a Strong Borel-Cantelli lemma for
recurrence.
Contrast the above result with
\cite[Theorem~\textup{4}]{barreira2001Hausdorff},
which states that \eqref{eq:returndim} holds when $\mu$ has
\emph{long return times}, i.e., for $\mu$-a.e.\ $x\in X$ and
sufficiently small $\varepsilon > 0$ 
\[
  \liminf_{r \to 0} \frac{\log (\mu 
  \{y\in B(x,r) : \tau_{B(x,r)}(y) \leq
\mu(B(x,r))^{-1+\varepsilon}\})}
  {\log \mu(B(x,r))} > 1.
\]
Furthermore, if $X$ is a compact smooth manifold, then
\cite[Theorem~\textup{5}]{barreira2001Hausdorff} states that if the
equilibrium measure $\mu$ is ergodic and supported on a locally maximum
hyperbolic set of a $C^{1+\alpha}$ diffeomorphism, then $\mu$ has long
return times and $\underline{R}(x) = \overline{R}(x) = \dim_{H}\mu$ for
$\mu$-a.e.\ $x$.
On the other hand, Barreira, Pesin and
Schmeling~\cite{barreira1999Dimension} prove that if $f$ is a
$C^{1+\alpha}$ diffeomorphism on a smooth compact manifold and $\mu$ a
hyperbolic $f$-invariant compactly supported ergodic Borel probability
measure, then
$\underline{d}_{\mu}(x) = \overline{d}_\mu(x) = \dim_H\mu$ 
for $\mu$-a.e.\ $x$.

\section{Properties of \texorpdfstring{$\mu$}{m} and 
  \texorpdfstring{$r_n$}{r\_n}}\label{sec:lemmas}
As previously mentioned, the functions $r_n \colon X \to [0,\infty)$ 
may not be well-defined since it may be the case that the measure $\mu$
assigns positive measure to the boundary of an open ball.
Lemma~\ref{lemma:r:existence} gives sufficient conditions on $x$,
$(M_n)$ and $\mu$ for $r_n(x)$ to be well-defined.
Lemma~\ref{lemma:r:uniform} gives conditions for which $r_n(x)$
converges uniformly to $0$, which allows us to use
Assumption~\ref{assumption:annulus}.
Lemma~\ref{lemma:partition} shows that we can use
Assumption~\ref{assumption:annulus} to control the
$\delta$-neighbourhood of the sets $A_k$ in \eqref{eq:X:partition}.
Recall that $(X,d)$ is a compact space and $(M_n)$ is a sequence in
$[0,1]$.

\begin{lemma}\label{lemma:r:existence}
  If $\mu$ has no atoms and satisfies $\mu(B) = \mu(\overline{B})$ for
  all open balls $B \subset X$, then for each $n \in \naturals$ and
  $x \in X$ there exists $r_n(x) \in [0,\infty)$ such that
  \begin{equation}\label{eq:rDefinition}
    \mu(B(x,r_n(x))) = M_n.
  \end{equation}
  Moreover, for each $n \in \naturals$ the function
  $r_n \colon X \to [0,\infty)$ is Lipschitz continuous with
  \[
    | r_n(x) - r_n(y) | \leq d(x,y)
  \]
  for all $x,y \in X$.
\end{lemma}

\begin{proof}
  Let $n \in \naturals$ and $x \in X$. The infimum
  \[
    r_n(x) = \inf \{r \geq 0 : \mu(B(x,r)) \geq M_n\}
  \]
  exists since the set contains $\diam X \in (0,\infty)$ and has a
  lower bound. 
  If $r_n(x) = 0$, then, since $\mu$ is non-atomic, it must be that
  $M_n = 0$, in which case $\mu(B(x,r_n(x))) = M_n$.
  If $r_n(x) > 0$, then for all $k > 0$ we have that
  $\mu(B(x,r_n(x) + \frac{1}{k})) \geq M_n$ and thus
  $\mu(\overline{B(x,r_n(x))}) \geq M_n$. Similarly,
  $\mu(B(x,r_n(x))) \leq M_n$. Hence, by using our assumption, we
  obtain that
  \[
    M_n \leq \mu(\overline{B(x,r_n(x))}) = \mu(B(x,r_n(x))) \leq M_n,
  \]
  which proves \eqref{eq:rDefinition}.

  To show that $r_n$ is Lipschitz continuous, let $x, y \in X$ and let
  $\delta = d(x,y)$. Then $B(y,r_n(y)) \subset B(x,r_n(y) + \delta)$
  and
  \[
    M_n = \mu(B(y,r_n(y))) \leq \mu(B(x,r_n(y) + \delta)).
  \]
  Hence, $r_n(x) \leq r_n(y) + \delta$. We conclude that
  $|r_n(x) - r_n(y)| \leq \delta$ by symmetry.
\end{proof}

The assumption $\mu(B) = \mu(\overline{B})$ for sufficiently small
radii follows from the Assumption~\ref{assumption:annulus}.

\begin{lemma}\label{lemma:r:uniform}
  If $\lim_{n \to \infty}M_n = 0$, then $\lim_{n \to \infty}r_n = 0$ 
  uniformly on $\supp \mu$.
\end{lemma}

\begin{proof}
  We first prove pointwise convergence on $\supp \mu$. Let $x\in X$ and
  suppose that $r_n(x)$ does not converge to $0$ as $n \to \infty$ for
  some $x \in X$. Then there exists $\varepsilon > 0$ and an increasing
  sequence $(n_i)$ such that $r_{n_i}(x) > \varepsilon$ and
  \[
    \mu( B(x,\varepsilon) ) \leq \mu( B(x,r_{n_i}(x)) ) = M_{n_i}
  \]
  for all $i$. Taking the limit as $i \to \infty$ gives us
  $\mu( B(x,\varepsilon) ) = 0$. Hence, $x \notin \supp \mu$.

  Now suppose that $r_n$ does not converge to $0$ uniformly on
  $\supp \mu$. Then there exist $\varepsilon > 0$, an increasing
  sequence $(n_i)$ and points $x_i \in \supp \mu$ such that
  $r_{n_i}(x_i) > \varepsilon$ for all $i > 0$. Since $\supp \mu$ is
  compact, we may assume that $x_i$ converges to some
  $x \in \supp \mu$. By Lipschitz continuity,
  \[
    r_{n_i}(x) \geq r_{n_i}(x_i) - d(x,x_i) 
    > \varepsilon - d(x,x_i).
  \]
  Thus,
  \[
    \lim_{i \to \infty} r_{n_i}(x) \geq \varepsilon,
  \]
  contradicting the fact that we have pointwise convergence to $0$ on
  $\supp\mu$.
\end{proof}

\begin{lemma}\label{lemma:partition}
  Suppose that $\mu$ satisfies Assumption~\ref{assumption:annulus} and
  that $(X,d)$ satisfies Assumption~\ref{assumption:packing}.
  Let $\rho < \min \{\varepsilon_0, \rho_0\}$ and consider the
  partition $\{A_k\}_{k=1}^{L}$ given in \eqref{eq:X:partition} where
  $\varepsilon = \frac{\rho}{2}$. Then for each $k$ and $\delta < \rho$
  \[
    \mu(A_k(\delta) \setminus A_k) \leqs \rho^{-K} \delta^{\alpha_0}.
  \]
\end{lemma}

\begin{proof}
  Let $0 < \delta < \rho \leq \rho_0$ and fix $k$.
  It suffices to prove that $A_k(\delta) \setminus A_k$ is contained in
  \[
    (B(x_k,\rho+\delta) \setminus B(x_k,\rho))
    \cup \bigcup_{i=1}^{k-1} 
    (B(x_i,\rho) \setminus B(x_i,\rho - \delta)).
  \]
  Indeed, since $\delta < \rho \leq \rho_0$, it then follows by
  Assumption~\ref{assumption:annulus} that
  \[
    \mu(A_k(\delta) \setminus A_k) \leqs k\delta^{\alpha_0}
    \leq L\delta^{\alpha_0}.
  \]
  Since $\rho < \varepsilon_0$, we have that $L \leqs \rho^{-K}$ by
  Assumption~\ref{assumption:packing} and we conclude.
  If $A_k(\delta) \setminus A_k$ is empty, then there is nothing to
  prove. So suppose that $A_k(\delta) \setminus A_k$ is non-empty and
  contains a point $x$. Suppose 
  $x \notin B(x_k,\rho + \delta) \setminus B(x_k,\rho)$. Since
  $x \in A_k(\delta) \subset B(x_k,\rho+\delta)$, we must have
  $x\in B(x_k,\rho)$. Furthermore, since $x\notin A_k$, by the
  construction of $A_k$, it must be that $x\in B(x_i,\rho)$ for some
  $i < k$. 
  If $x \in B(x_i,\rho-\delta)$, then, since $x \in A_k(\delta)$, there 
  exists $y\in A_k$ such that $d(x,y) < \delta$. Hence,
  $y \in B(x_i,\rho)$ and $y\notin A_k$, a contradiction. Thus 
  $x\in B(x_i,\rho) \setminus B(x_i, \rho - \delta)$, which concludes
  the proof.
\end{proof}

\section{Measure and Correlations of \texorpdfstring{$E_n$}{E\_n}}
\label{sec:props}
We give measure and correlations estimates of $E_n$ in
Proposition~\ref{prop:main3}, which will follow from
Propositions~\ref{prop:main1} and \ref{prop:main2}.
The proofs of those propositions are modifications of the proofs of
Lemmas~\textup{3.1} and \textup{3.2} in \cite{kirsebom2021Shrinking}.
The modifications account for the more difficult setting; that is,
working with a general metric space $(X,d)$ and for assuming decay of
correlation for a smaller class of observables, i.e., for Lipschitz
continuous functions rather than $L_1$ and $BV$ functions.

Throughout the section, $(M_n)$ denotes a sequence in $[0,1]$ that
converges to $0$ and satisfies Assumption~\ref{assumption:seq}.
We also assume that the space $(X,\mu,T,d)$ satisfies
Assumptions~\ref{assumption:ball}--\ref{assumption:packing}.
Hence, the functions $r_n$ defined by $\mu(B(x,r_n(x))) = M_n$ are
well-defined for all $n$ by Lemma~\ref{lemma:r:existence} and converge
uniformly to $0$ by Lemma~\ref{lemma:r:uniform}. Thus, for large enough
$n$ the functions $r_n$ are uniformly bounded by $\rho_0$ on
$\supp\mu$. For notational convenience, $C$ will denote an arbitrary
positive constant depending solely on $(X,\mu,T,d)$ and may vary from
equation to equation.

In the following propositions, it will be useful to partition the space
into sets whose diameters are controlled. Since we wish to work with
Lipschitz continuous functions in order to apply decay of correlations,
we also approximate the indicator functions on the elements of the
partition.
Let $\rho < \min \{\varepsilon_0, \rho_0\}$ and recall the construction
of the partition $\{A_k\}_{k=1}^{L}$ of $X$ in \eqref{eq:X:partition}
where $\varepsilon = \frac{\rho}{2}$.
For $\delta \in (0,\rho)$ define the functions
$\{h_k \colon X \to [0,1]\}_{k=1}^{L}$ by
\begin{equation}\label{eq:h:definition}
  h_k(x) = \min \{1, \delta^{-1} \dist(x, X \setminus A_{k}(\delta))\}.
\end{equation}  
Notice that the functions are Lipschitz continuous with Lipschitz
constant $\delta^{-1}$.

\begin{proposition}\label{prop:main1}
  With the above assumptions on $(X,\mu,T,d)$ and $(M_n)$, suppose
  additionally that $2$-fold correlations decay exponentially for
  Lipschitz continuous observables. Let $n\in\naturals$ and consider
  the function $F \colon X \times X \to [0,1]$ given by
  \[
    F(x,y) =
    \begin{cases}
      1 & \text{if $x \in B(y, r_n(y))$}, \\
      0 & \text{otherwise}.
    \end{cases}
  \]
  Then there exists constants $C,\eta > 0$ independent of $n$ such that
  \begin{equation}\label{ineq:main1}
    \Bigl| \int F(T^{n}x, x) \diff\mu(x) - M_n \Bigr|
    \leq C e^{-\eta n}.
  \end{equation}
\end{proposition}

\begin{proof}
We prove the upper bound part of \eqref{ineq:main1} as the lower bound
follows similarly.

Consider $Y = \{(x,y) : F(x,y) = 1\}$. To establish an upper bound
define $\hat{F} \colon X \times X \to [0, 1]$ as follows. 
Let $\varepsilon_n \in (0,\rho_0]$ decay exponentially as $n\to\infty$
and define 
\[
  \hat{F}(x,y) = \min \bigl\{1, \varepsilon_n^{-1} 
  \dist\bigl((x,y), (X\times X) \setminus Y(\varepsilon_n)\bigr)
  \bigr\}.
\]

Note that it suffices to prove the result for sufficiently large $n$.
We also assume that $\varepsilon_0,\rho_0 < 1$, as we may replace them
by smaller values otherwise. Therefore, by Lemma~\ref{lemma:r:uniform},
we may assume that $r_n(x) \leq \rho_0$ for all $x\in \supp\mu$.
If $M_n = 0$, then $r_n = 0$ on $\supp\mu$ and we find that we do not
have any further restrictions on $\varepsilon_n$.
If $M_n > 0$, then by Assumption~\ref{assumption:ball}
\[
  r_n(x) \geqs M_n^{1/s} > 0.
\]
Thus, by using Assumption~\ref{assumption:seq} and the exponential
decay of $\varepsilon_n$, we may assume that $3\varepsilon_n < r_n(x)$
for all $x\in X$. We continue with assuming that $M_n > 0$ since the
arguments are significantly easier when $M_n = 0$.

Now, $\hat{F}$ is Lipschitz continuous with Lipschitz constant 
$\varepsilon_n^{-1}$. Furthermore, 
\begin{equation}\label{prop1:bound:Fhat}
  F(x,y) \leq \hat{F}(x,y) 
  \leq \charfun_{B(y,r_n(y) + 3\varepsilon_n)}(x).
\end{equation}
This is because if $\hat{F}(x,y) > 0$, then $d((x,y),(x',y')) <
\varepsilon_n$
for some $(x',y') \in Y$ and so
\[
  \begin{split}
    d(x,y) &\leq d(x',y') + d(x,x') + d(y',y) \\
    &< r_n(y') + 2\varepsilon_n \\
    &\leq r_n(y) + 3\varepsilon_n,
  \end{split}
\]
where the last inequality follows from the Lipschitz continuity of
$r_n$. Hence,
\[
  \int F(T^{n}x,x) \diff\mu(x) \leq \int \hat{F}(T^{n}x,x) \diff\mu(x).
\]
Using a partition of $X$, we construct an approximation $H$ of
$\hat{F}$ of the form
\[
  \sum \hat{F}(x,y_{k}) h_{k}(y),
\]
where $h_{k}$ are Lipschitz continuous, for the purpose of leveraging
$2$-fold decay of correlations.
Let $\rho_n < \min \{\varepsilon_0, \rho_0\}$ and consider the
partition $\{A_k\}_{k=1}^{L_n}$ of $X$ given in \eqref{eq:X:partition}
for $\varepsilon = \rho_n$. Let $\delta_n < \rho_n$ and obtain the
functions $h_k$ as in \eqref{eq:h:definition} for $\delta = \delta_n$.
Let 
\[
  I = \{k \in [1,L_n] : A_k \cap \supp\mu \neq \emptyset\}
\]
and $y_k \in A_k \cap \supp\mu$ for $k \in I$. Note that
$|I|\leq L_n \leqs \rho_n^{-K}$ by Assumption~\ref{assumption:packing}
since $\rho_n < \varepsilon_0$. Define $H \colon X \times X \to \reals$
by
\[
  H(x,y) = \sum_{k\in I} \hat{F}(x,y_k)h_k(y).
\]
For future reference, note that
\[
  \sum_{k \in I} h_k \leq \sum_{k \in I} (\charfun_{A_k}
    + \charfun_{A_k(\delta_n) \setminus A_k} )
  \leq 1 + \sum_{k \in I} \charfun_{A_k(\delta_n) \setminus A_k},
\]
which implies that
\[
  \int \sum_{k \in I} h_k(x) \diff\mu(x)
  \leq 1 + \sum_{k \in I} \mu(A_k(\delta_n) \setminus A_k).
\]
Using Lemma~\ref{lemma:partition}, we obtain that
\begin{equation}\label{prop1:sum:f}
  \int \sum_{k \in I} h_k(x) \diff\mu(x)
  \leq 1 + C \rho_n^{-2K}\delta_n^{\alpha_0}.
\end{equation}
Now,
\begin{equation}\label{prop1:boundH}
  \begin{split}
    \int \hat{F}(T^{n}x,x) \diff\mu(x) 
    \leq &\int H(T^{n}x,x) \diff\mu(x) \\
    &+ \int |\hat{F}(T^{n}x,x) - H(T^{n}x,x)| \diff\mu(x).
  \end{split}
\end{equation}

We use decay of correlations to bound the first integral from
above:
\[
  \begin{split}
    \int \hat{F}(T^{n}x,y_k)h_k(x) \diff\mu(x)
    \leq &\int \hat{F}(x,y_k)\diff\mu(x) \int h_k(x) \diff\mu(x) \\
    &+ C\lipnorm{\hat{F}}\lipnorm{h_k}e^{-\tau n}.
  \end{split}
\]
Notice that by \eqref{prop1:bound:Fhat},
\[
  \int \hat{F}(x,y_k)\diff\mu(x) 
  \leq \mu( B(y_k, r_n(y_k) + 3\varepsilon_n) ).
\]
Since $y_k \in \supp \mu$, we have that
$3\varepsilon_n < r_n(y_k) \leq \rho_0$. Hence, by
Assumption~\ref{assumption:annulus},
\begin{equation}\label{prop1:bound:Mn}
  \mu( B(y_k, r_n(y_k) + 3\varepsilon_n) )
  \leq M_n + C\varepsilon_n^{\alpha_0}.
\end{equation}
Furthermore, $\lipnorm{\hat{F}} \leqs \varepsilon_n^{-1}$ and 
$\lipnorm{h_k} \leqs \delta_n^{-1}$. Hence,
\[
  \int \hat{F}(T^{n}x,y_k)h_k(x) \diff\mu(x)
  \leq (M_n + C\varepsilon_n^{\alpha_0}) \int h_k(x) \diff\mu(x)
  + C\varepsilon_n^{-1}\delta_n^{-1}e^{-\tau n}
\]
and
\[
  \int H(T^{n}x,x) \diff\mu(x) 
  \leq (M_n + C\varepsilon_n^{\alpha_0})\sum_{k \in I} 
    \int h_k(x)\diff\mu(x)
  + C \sum_{k\in I} \varepsilon_n^{-1} 
    \delta_n^{-1}e^{-\tau n}.
\]
In combination with \eqref{prop1:sum:f} and
$|I| \leqs \rho_n^{-K}$, the former equation gives
\begin{equation}\label{prop1:boundH:1}
  \int H(T^{n}x,x) \diff\mu(x) 
  \leq (M_n + C\varepsilon_n^{\alpha_0})
    (1 + C\rho_n^{-2K}\delta_n^{\alpha_0})
  + C \varepsilon_n^{-1} \delta_n^{-1}\rho_n^{-K}e^{-\tau n}.
\end{equation}

We now establish a bound on 
\[
  \int | \hat{F}(T^{n}x,x) - H(T^{n}x,x) |\diff\mu(x)
\]
in \eqref{prop1:boundH}. By the triangle inequality, 
\[
  |\hat{F}(T^{n}x,x) \charfun_{A_k}(x) - \hat{F}(T^{n}x,y_k) h_k(x)| 
\]
is bounded above by 
\[
  | \hat{F}(T^{n}x,x) - \hat{F}(T^{n}x,y_k) | \charfun_{A_k}(x)
  + | \charfun_{A_k}(x) - h_k(x) | \hat{F}(T^{n}x,y_k).
\]
Now, if $x \in A_k$, then by Lipschitz continuity 
\[
  | \hat{F}(T^{n}x,x) - \hat{F}(T^{n}x,y_k) | 
  \leq \varepsilon_n^{-1} d( (T^{n}x, x), (T^{n}x, y_k) ) 
  \leqs \rho_n \varepsilon_n^{-1}.
\]
Also, using the fact that $\hat{F} \leq 1$ 
and $\charfun_{A_k} \leq h_k \leq \charfun_{A_k(\delta_n)}$,
we obtain 
\[
  | \hat{F}(T^{n}x,x) \charfun_{A_k}(x) 
    - \hat{F}(T^{n}x,y_k) h_k(x) | 
  \leq C\rho_n \varepsilon_n^{-1} \charfun_{A_k}(x) 
    + \charfun_{A_k(\delta_n) \setminus A_k}(x).
\]
Hence, since 
$\hat{F}(x,y) = \sum_{k \in I} \hat{F}(x,y)\charfun_{A_k}(y)$ for
$y \in \supp\mu$,
we obtain that for $x \in \supp\mu$
\[
  | \hat{F}(T^{n}x,x) - H(T^{n}x,x) | 
\]
is bounded above by
\[
  \begin{split}
    &\sum_{k \in I} | \hat{F}(T^{n}x,x) \charfun_{A_k}(x)
      - \hat{F}(T^{n}x,y_k) h_k(x) | \\
    &\qquad \leq \sum_{k \in I} 
      \left( C \varepsilon_n^{-1} \rho_n \charfun_{A_k}(x) 
      + \charfun_{A_k(\delta_n) \setminus A_k}(x) \right).
  \end{split}
\]
Now, $\sum_{k \in I} \charfun_{A_k} = 1$ on $\supp\mu$ gives
\begin{equation}\label{prop1:boundH:sum1}
  |\hat{F}(T^{n}x,x) - H(T^{n}x,x)| 
  \leq C \varepsilon_n^{-1} \rho_n + 
    \sum_{k \in I} \charfun_{A_k(\delta_n) \setminus A_k}(x)
\end{equation}
for $x\in\supp\mu$ and thus
\[
  \int | \hat{F}(T^{n}x,x) - H(T^{n}x,x) | \diff\mu(x)
  \leq C \varepsilon_n^{-1} \rho_n +
  \sum_{k \in I} 
    \mu(A_k(\delta_n) \setminus A_k).
\]
By Lemma~\ref{lemma:partition},
\begin{equation}\label{prop1:boundH:2}
  \int | \hat{F}(T^{n}x,x) - H(T^{n}x,x) | \diff\mu(x)
  \leqs \varepsilon_n^{-1} \rho_n + \rho_n^{-2K} \delta_n^{\alpha_0}.
\end{equation}
By combining equations \eqref{prop1:boundH}, \eqref{prop1:boundH:1} and
\eqref{prop1:boundH:2}, we obtain that
\[
  \int \hat{F}(T^{n}x,x)\diff\mu(x) - M_n
  \leqs \varepsilon_n^{\alpha_0} + \rho_n^{-2K}\delta_n^{\alpha_0}
  + \varepsilon_n^{-1}\rho_n 
  + \varepsilon_n^{-1}\delta_n^{-1}\rho_n^{-2K}e^{-\tau n}.
\]
Now, let
\begin{align*}
  \varepsilon_n &= e^{-\gamma n}, \\
  \rho_n &= \varepsilon_n^2, \\
  \delta_n &= \varepsilon_n^{2 + 4K/\alpha_0},
\end{align*}
where 
$\gamma 
= \frac{1}{2}\bigl(3 + 4K + \frac{4K}{\alpha_0}\bigr)^{-1} \tau$.
Then $\varepsilon_n$ decays exponentially as $n \to \infty$ and
$\delta_n < \rho_n$ as required. For large enough $n$, we have
$3\varepsilon_n < r_n(x)$ and $\rho_n < \min \{\varepsilon_0, \rho_0\}$
as we assumed. Finally,
\[
  \eta = \min \Bigl\{\frac{\tau}{2}, \gamma, \alpha_0\gamma\Bigr\}
\]
gives us an upper bound
\[
  \int F(T^{n}x,x)\diff\mu(x) - M_n \leqs e^{-\eta n}.
\]
\end{proof}

\begin{proposition}\label{prop:main2}
  With the above assumptions on $(X,\mu,T,d)$ and $(M_n)$, suppose
  additionally that $3$-fold correlations decay exponentially for
  Lipschitz continuous observables. Let $m,n \in \naturals$ and
  consider the function $F\colon X\times X\times X \to [0,1]$ given by
  \[
    F(x,y,z) =
    \begin{cases}
      1 & \text{if $x\in B(z, r_{n+m}(z))$ and $y\in B(z, r_n(z))$,} \\
      0 & \text{otherwise}.
    \end{cases}
  \]
  Then, there exists $C,\eta > 0$ independent of $n$ and $m$ such that
  \[
    \int F( T^{n + m}x, T^{n}x, x) \diff\mu(x)
    \leq M_n M_{n+m} + C( e^{-\eta n} + e^{- \eta m} ).
  \]
\end{proposition}

The proof of Proposition~\ref{prop:main2} follows roughly that of
Proposition~\ref{prop:main1} by making some modifications to account
for the added dimension.

\begin{proof}
Again, we assume w.l.o.g.\ that $n$ is sufficiently large and that
$\varepsilon_0,\rho_0 < 1$ and $M_n > 0$.

Notice that $F(x,y,z) = G_1(x,z) G_2(y,z)$ for
\[
  G_1(x,z) = 
  \begin{cases}
    1 & \text{if $x \in B(z, r_{n+m}(z))$}, \\
    0 & \text{otherwise},
  \end{cases}
\]
and
\[
  G_2(y,z) = 
  \begin{cases}
    1 & \text{if $y \in B(z, r_{n}(z))$}, \\
    0 & \text{otherwise}.
  \end{cases}
\]
Let $\varepsilon_{m,n} < \rho_0$ converge exponentially to $0$ as 
$\min \{m,n\} \to \infty$. 
Define the functions
$\hat{G}_1, \hat{G}_2 \colon X \times X \to [0,1]$ similarly to how
$\hat{F}$ was defined in Proposition~\ref{prop:main1}. That is,
$\hat{G}_1$ and $\hat{G}_2$ are Lipschitz continuous with Lipschitz
constant $\varepsilon_{m,n}^{-1}$ and satisfy
\[
  G_1(x,y) \leq \hat{G}_1(x,y) \leq
  \charfun_{B(y, r_{n+m}(y) + 3\varepsilon_{m,n})}(x)
\]
and
\[
  G_2(x,y) \leq \hat{G}_2(x,y) \leq
  \charfun_{B(y, r_{n}(y) + 3\varepsilon_{m,n})}(x).
\]
Thus 
\begin{equation}\label{prop2:ineq1}
  \int F( T^{n + m}x, T^{n}x, x) \diff\mu(x)
  \leq \int \hat{G}_1(T^{n + m}x,x) \hat{G}_2(T^{n}x, x) \diff\mu(x).
\end{equation}
Let $\rho_{m,n} < \min \{\varepsilon_0, \rho_0\}$ and consider
the partition $\{A_k\}_{k=1}^{L_{m,n}}$ of $X$ given in
\eqref{eq:X:partition} for $\varepsilon = \frac{\rho_{m,n}}{2}$. Let
$\delta_{m,n} < \rho_{m,n}$ and obtain the functions $h_k$ as in
\eqref{eq:h:definition} for $\delta = \delta_{m,n}$. Let
\[
  I = \{k \in [1,L_{m,n}] : A_k \cap \supp\mu \neq \emptyset\}
\]
and $y_k \in A_k \cap \supp\mu$ for $k \in I$. Define the functions
$H_1, H_2 \colon X \times X \to \reals$ by
\[
  H_i(x,y) = \sum_{k \in I} \hat{G}_i(x, y_k) h_k(x)
\]
for $i = 1,2$.
Using the triangle inequality, we bound the right-hand side of 
\eqref{prop2:ineq1} by 
\begin{equation}\label{prop2:bound1}
  \begin{split}
    &\int |\hat{G}_1(T^{n+m}x,x) \hat{G}_2(T^{n}x,x) 
      - H_1(T^{n+m}x,x) H_2(T^{n}x,x)| \diff\mu(x) \\
    &\qquad + \int H_1(T^{n+m}x,x) H_2(T^{n}x,x) \diff\mu(x)
  \end{split}
\end{equation}
We look to bound the first integral. Using the triangle inequality and
the fact that $\hat{G}_1 \leq 1$, we obtain that the integrand is
bounded above by
\[
  | \hat{G}_2(T^{n}x,x) - H_2(T^{n}x,x) |
  + H_2(T^{n}x,x) | \hat{G}_1(T^{n+m}x,x) - H_1(T^{n+m}x,x) |.
\]
As shown in \eqref{prop1:boundH:2} 
\[
  \int | \hat{G}_2(T^{n}x,x) - H_2(T^{n}x,x) | \diff\mu(x)
  \leqs \varepsilon_{m,n}^{-1} \rho_{m,n} 
    + \rho_{m,n}^{-2K} \delta_{m,n}^{\alpha_0}
\]
and by \eqref{prop1:boundH:sum1}
\[
  | \hat{G}_1(T^{n+m}x,x) - H_1(T^{n+m}x,x) | 
  \leq C \varepsilon_{m,n}^{-1} \rho_{m,n} 
  + \sum_{k \in I} 
    \charfun_{A_k(\delta_{m,n}) \setminus A_k}(x)
\]
and
\[
  H_2(T^{n}x,x) \leq 1 + 
  \sum_{k \in I} \charfun_{A_k(\delta_{m,n}) \setminus A_k}(x).
\]
Thus
\[
  \int H_2(T^{n}x,x) | \hat{G}_1(T^{n+m}x,x) - H_1(T^{n+m}x,x) | 
  \diff\mu(x)
\]
is bounded above by
\[
  \begin{split}
    &C \varepsilon_{m,n}^{-1} \rho_{m,n} 
    + C(1 + \varepsilon_{m,n}^{-1} \rho_{m,n}) 
    \sum_{k \in I} 
      \charfun_{A_k(\delta_{m,n}) \setminus A_k}(x) \\
    &\qquad + \sum_{k \in I} \sum_{k' \in I} 
      \charfun_{A_k(\delta_{m,n}) \setminus A_k}(x)
      \charfun_{A_{k'}(\delta_{m,n}) \setminus A_{k'}}(x).
  \end{split}
\]
Integrating and using Lemma~\ref{lemma:partition} together with the
trivial estimate $\mu(A \cap B) \leq \mu(A)$, we obtain that
\begin{equation}\label{prop2:bound:term1}
  \begin{split}
    &\int H_2(T^{n}x,x) 
      \bigl| \hat{G}_1(T^{n+m}x,x) - H_1(T^{n+m}x,x) \bigr| 
    \diff\mu(x) \\
    &\qquad \leqs \varepsilon_{m,n}^{-1} \rho_{m,n} 
    + (1 + \varepsilon_{m,n}^{-1} \rho_{m,n})
      \rho_{m,n}^{-2K} \delta_{m,n}^{\alpha_0}
    + \rho_{m,n}^{-3K} \delta_{m,n}^{\alpha_0}.
  \end{split}
\end{equation}

We now estimate the second integral in \eqref{prop2:bound1} using decay
of correlations:
\[
  \int H_1(T^{n+m}x,x) H_2(T^{n}x,x) \diff\mu(x) 
\]
is bounded above by
\begin{equation}\label{prop2:boundH:3}
  \begin{split}
    &\sum_{k \in I} \sum_{k' \in I}
      \norm{\hat{G}_1(\cdot, y_k)}_{L^{1}}
      \norm{\hat{G}_2(\cdot, y_{k'})}_{L^{1}}
      \norm{h_k h_{k'}}_{L^{1}} \\
    &\qquad + \sum_{k \in I} \sum_{k' \in I} 
      C \lipnorm{\hat{G}_1} \lipnorm{\hat{G}_2}
      \lipnorm{h_k h_{k'}}
      ( e^{-\tau n} + e^{-\tau m} ).
  \end{split}
\end{equation}
Similarly to how \eqref{prop1:sum:f} and \eqref{prop1:bound:Mn} were
established in Proposition~\ref{prop:main1}, we obtain the bounds
\begin{align*}
  \norm{\hat{G}_1(\cdot,y_k)}_{L^{1}}
  &\leq M_{n+m} + C\varepsilon_{m,n}^{\alpha_0}, \\
  \norm{\hat{G}_2(\cdot,y_k)}_{L^{1}}
  &\leq M_{n} + C\varepsilon_{m,n}^{\alpha_0}
\end{align*}
and
\[
  \sum_{k \in I} \sum_{k' \in I} 
  \norm{h_k h_{k'}}_{L^{1}} 
  \leq 1 + C \rho_{m,n}^{-3K} \delta_{m,n}^{\alpha_0}.
\]
Also,
\[
  \lipnorm{\hat{G}_1} = \lipnorm{\hat{G}_2}
  \leqs \varepsilon_{m,n}^{-1}
\]
and
\[
  \lipnorm{h_k h_{k'}} 
  \leqs \delta_{m,n}^{-1}.
\]
Combining the previous equations with 
\eqref{prop2:boundH:3}, we obtain the upper bound
\begin{equation}\label{prop2:bound:term2}
  \begin{split}
    \int H_1(T^{n+m}x,x) H_2(T^{n}x,x) \diff\mu(x) 
    \leq & M_{n+m}M_n + C\varepsilon_{m,n}^{\alpha_0} 
      + C\rho_{m,n}^{-3K}\delta_{m,n} ^{\alpha_0} \\
    &+ C\rho_{m,n}^{-2K}\varepsilon_{m,n}^{-2}\delta_{m,n}^{-1}
      (e^{-\tau n} + e^{-\tau m}).
  \end{split}
\end{equation}
Combining and simplifying equations \eqref{prop2:bound:term1} and
\eqref{prop2:bound:term2}, we obtain that 
\[
  \begin{split}
    &\int F(T^{n+m}x,T^{n}x,x) \diff\mu(x) - M_{n+m}M_n
    \leqs \varepsilon_{m,n}^{-1} \rho_{m,n} 
      + \varepsilon_{m,n}^{-1} \rho_{m,n}^{1-2K}
      \delta_{m,n}^{\alpha_0} \\
    &\qquad+ \rho_{m,n}^{-3K} \delta_{m,n}^{\alpha_0}
      + \varepsilon_{m,n}^{\alpha_0}
    + \rho_{m,n}^{-2K}\varepsilon_{m,n}^{-2}\delta_{m,n}^{-1}
      (e^{-\tau n} + e^{-\tau m}).
  \end{split}
\]
Now, let
\begin{align*}
  \varepsilon_{m,n} &= e^{-\gamma \min \{m,n\}}, \\
  \rho_{m,n} &= \varepsilon_{m,n}^2, \\
  \delta_{m,n} &= \varepsilon_{m,n}^{2 + 6K/\alpha_0},
\end{align*}
where
$\gamma
= \frac{1}{2}\bigl(4 + 4K + \frac{6K}{\alpha_0}\bigr)^{-1} \tau$.
Then $\varepsilon_{m,n}$ decays exponentially as $n \to \infty$ and
$\delta_{m,n} < \rho_{m,n}$ as required. For large enough $n$, we have
$3\varepsilon_{m,n}< r_n(x)$ and $\rho_{m,n} < \min \{\varepsilon_0,
\rho_0\}$ as we assumed. Finally,
\[
  \eta = \min \Bigl\{ \frac{\tau}{2}, \gamma, \alpha_0\gamma,
  (1 + 2\alpha_0 + 2K)\gamma \Bigr\}
\]
gives us the upper bound
\[
  \int F(T^{n+m}x,T^{n}x,x) \diff\mu(x) - M_{n+m}M_n
  \leqs e^{-\eta n} + e^{-\eta m}.
\]
\end{proof}

\begin{proposition}\label{prop:main3}
  With the above assumptions on $(X,\mu,T,d)$ and $(M_n)$, suppose
  additionally that $3$-fold correlations decay exponentially for
  Lipschitz continuous observables. There exists $C,\eta > 0$ such that
  for all $m,n \in \naturals$
  \[
    | \mu(E_n) - M_n | \leq C e^{-\eta n}
  \]
  and
  \[
    \mu(E_{n+m} \cap E_{n}) \leq \mu(E_{n+m})\mu(E_n)
    + C(e^{-\eta n} + e^{-\eta m}). 
  \]
\end{proposition}

\begin{proof}
  The first inequality follows directly from
  Proposition~\ref{prop:main1} by noticing that
  $\mu(E_n) = \int F(T^{n}x,x)\diff\mu(x)$ and the second inequality
  follows from the first in combination with
  Proposition~\ref{prop:main2} by noticing that 
  $\mu(E_{n+m} \cap E_{n}) = \int F(T^{n+m}x,T^{n}x,x)\diff\mu(x)$.
\end{proof}

\section{Proof of Theorems~\ref{thm:main} and \ref{thm:general}}
\label{sec:thm:proof}
We can now prove our main results.

\begin{proof}[Proof of Theorem~\ref{thm:general}]
  By Proposition~\ref{prop:main3}, we see that the conditions of
  Theorem~\ref{thm:persson} are met and we conclude the proof.
\end{proof}

Recall that $M$ is a compact smooth $N$-dimensional manifold, and that
$d$ and $\vol$ are the induced metric and volume measure. The
injectivity radius of $M$ is denoted by $\inj_M$ and is positive.

\begin{proof}[Proof of Theorem~\ref{thm:main}]
  By \cite[Proposition~3.1]{kotani2001Pressure}, the system has
  $3$-fold decay of correlations for Lipschitz continuous observables.
  Thus, the theorem follows from Theorem~\ref{thm:general} once we
  verify that Assumptions~\ref{assumption:annulus} and
  \ref{assumption:packing} hold.

  For Assumption~\ref{assumption:packing}, let 
  $\varepsilon \in (0,\inj_M)$ and consider a maximal
  $\varepsilon$-packing $\{B(x_k,\varepsilon)\}_{k=1}^{L}$ of $M$,
  i.e., $L = P_\varepsilon(M)$. Now,
  \[
    \begin{split}
      L \min_{k} \vol( B(x_k,\varepsilon) )
      &\leq \sum_{k=1}^{L} \vol( B(x_k,\varepsilon) ) \\
      &= \vol\Bigl( \bigcup_{k=1}^{L} B(x_k,\varepsilon) \Bigr) \\
      &\leq \vol(M).
    \end{split}
  \]
  By \cite[Proposition~14]{croke1980Isoperimetric}, there exists
  $C_1 >0$ such that
  \begin{equation}\label{eq:croke}
    \vol(B(y,\varepsilon)) \geq C_1 \varepsilon^{N}
  \end{equation}
  for all $y\in M$. Hence,
  \[
    L \leq \frac{\vol(M)}{\min_{k}\vol( B(x_k,\varepsilon) )}
    \leqs \varepsilon^{-N}.
  \]
  Thus, Assumption~\ref{assumption:packing} holds for $K = N$ and
  $\varepsilon_0 = \inj_M$.
  
  For Assumption~\ref{assumption:annulus}, let $x\in M$ and
  $0 < \varepsilon < \rho \leq \inj_M$. For convenience, let
  $\varepsilon' = \frac{\varepsilon}{2}$. Consider a maximal
  $\varepsilon'$-packing $\{B(x_k,\varepsilon')\}_{k=1}^{L}$ of
  the annulus $\{y : \rho \leq d(x,y) < \rho + \varepsilon\}$.
  Then $\{B(x_k,\varepsilon')\}_{k=1}^{L}$ is contained in the set
  $\{y : \rho - \varepsilon' \leq d(x,y) < \rho + 3\varepsilon'\}$.
  Thus,
  \[
    \begin{split}
      L\min_k \vol(B(x_k, \varepsilon'))
      &\leq \sum_{i=1}^{n} \vol(B(x_k, \varepsilon')) \\
      &= \vol \Bigl( \bigcup_{i=1}^{n} B(x_k, \varepsilon') \Bigr)  \\
      &\leq \vol \{y \in M : \rho - \varepsilon' 
        \leq d(x,y) < \rho + 3\varepsilon'\} \\
      &= \vol(B(x, \rho + 3\varepsilon'))
        - \vol(B(x, \rho - \varepsilon'))
    \end{split}
  \]
  Since $M$ is compact, there exists $C_2 > 0$ such that for all
  $0 < s < r$ and $y \in M$
  \[
    \vol ( B(y, r) ) - \vol ( B(y, s) )
    \leq C_2(r-s)
  \]
  (see for instance \cite[p127]{chavel2006Riemannian}).
  Therefore,
  \[
    L\min_k \vol(B(x_k, \varepsilon')) \leqs 4\varepsilon'
    \leqs \varepsilon.
  \]
  Using \eqref{eq:croke} again and noting that
  $\varepsilon' < \varepsilon < \inj_M$, we obtain that
  \[
    L \leqs \frac{\varepsilon}{\varepsilon^{N}} = \varepsilon^{1-N}.
  \]
  Now, $\{B(x_k,\varepsilon)\}_{k=1}^{L}$ covers
  $\{y \in M : \rho \leq d(x,y) < \rho + \varepsilon\}$. Hence,
  \[
    \begin{split}
      \mu \{y \in M : \rho \leq d(x,y) < \rho + \varepsilon\}
      &\leq \mu\Bigl(\bigcup_{k=1}^{L}B(x_k,\varepsilon)\Bigr) \\
      &\leq L\max_{k} \mu(B(x_k,\varepsilon)) \\
      &\leqs \varepsilon^{1-N} \varepsilon^{s} \\
      &= \varepsilon^{1+s-N},
    \end{split}
  \]
  where in the third inequality we have used that
  $\mu(B(x_k,\varepsilon)) \leqs \varepsilon^{s}$.
  Thus, Assumption~\ref{assumption:annulus} holds for
  $\rho_0 = \inj_M$ and $\alpha_0 = 1 + s - N$. Note that
  $\alpha_0 > 0$ as we assumed that $s > N - 1$.
\end{proof}

\section{Final Remarks}\label{sec:finalremarks}

\subsection{}\label{final:background}
By restricting the underlying system, one can improve Boshernitzan's
general result.
Pawelec~\cite{pawelec2017Iterated} proved that
\[
  \liminf_{n \to \infty} (n \log \log n)^{1/\beta} d(T^{n}x,x) = 0
\]
for $\mu$-a.e.\ $x$ under the assumptions of exponential decay of
correlations and some assumptions on $\mu$ related to $\beta$.
For self-similar sets, Chang, Wu and Wu~\cite{chang2019Quantitative}
and Baker and Farmer~\cite{baker2021Quantitative} obtained a dichotomy
result for the set
\[
  R = R(T, (r_n)) = \{x \in X : d(T^{n}x,x) < r_n 
    \text{ for infinitely many $n \in \naturals$}\}
\]
to have zero or full measure depending on the convergence or divergence
of a series involving $(r_n)$.
For mixing systems on $[0,1]$, Kirsebom, Kunde and
Persson~\cite{kirsebom2021Shrinking} obtained several results, with
varying conditions on $(r_n)$, on the measure of $R$ and related sets.
For integer matrices on the $N$-dimensional torus such that no
eigenvalue is a root of unity, they obtain a condition for $R$ to have
zero or full measure, again depending on the convergence of a series
involving $(r_n)$.
He and Liao~\cite{he2023Quantitative} extend Kirsebom, Kunde and
Persson's results to non-integer matrices and to targets that are
hyperrectangles or hyperboloids instead of open balls.
Allen, Baker and B\'ar\'any~\cite{allen2022Recurrence} give sufficient
conditions for the set $R$, adapted to subshifts of finite type, to
have measure $1$ and sufficient conditions for the set to have measure
$0$. They also obtain a critical threshold for which the measure of $R$
transitions from $0$ to $1$.

\subsection{}\label{final:applications}
If $\mu$ is absolutely continuous with respect to the Lebesgue measure
$\lambda$ with density in $L^{p}$, then $\mu$ satisfies
Assumptions~\ref{assumption:ball} and \ref{assumption:annulus}. Namely,
if $\diff\mu = h \diff\lambda$, then
\[
  \mu(A) = \int_{A} h \diff\lambda
  \leq \norm{h}_{L^{p}} \lambda(A)^{1/q},
\]
where $q$ is the harmonic conjugate of $p$.
Since
$\lambda(B(x,\rho + \varepsilon) \setminus B(x,\rho) )
\leq \varepsilon^{\alpha_0}$ 
and
$\lambda(B(x,\rho)) \leq \rho^{s}$
for some $\alpha_0$ and $s$,
we see that $\mu$ also satisfies the assumptions.
Gupta, Holland and Nicol \cite{gupta2011Extreme} consider planar
dispersing billiards. These systems have an absolutely continuous
invariant measure and so Assumptions~\ref{assumption:ball} and
\ref{assumption:annulus} are satisfied. Furthermore, some of the
systems have exponential decay of $2$-fold correlations for H\"older
observables as shown by Young \cite{young1998Statistical} and
exponential decay of multiple decorrelation as shown in
\cite[Theorem~7.41]{chernov2006Chaotic}. Thus our results hold for
such systems.
The authors in \cite{gupta2011Extreme} also consider Lozi maps, for
which a broad class satisfies Assumptions~\ref{assumption:annulus} and
have exponential decay of $2$-fold correlations for H\"older
observables.
For our result to apply to these systems, one would have to show that
$3$-fold correlations decay exponentially as well.

\subsection{}\label{final:improvements}
The main limitation of the method of proof is the reliance on
Assumption~\ref{assumption:annulus}, namely that we require control
over the measure of $\delta$-neighbourhoods of sets.
It is reasonable to think that the result should hold for measures that
do not satisfy this assumption; however, in that case, instead of
considering open balls, one would need to consider other shrinking
neighbourhoods. The reason for this is that for a general measure $\mu$
there exists $M_n \in [0,1]$ and $x\in X$ such that there is no open
ball $B_n(x)$ with $\mu(B_n(x)) = M_n$.
Furthermore, one would like to improve the condition
$M_n \geq n^{-1} (\log n)^{4+\varepsilon}$ in
Assumption~\ref{assumption:seq} to one closer to the critical decay
rate $M_n \geq n^{-1}$.

\subsection*{Acknowledgements}
I wish to thank my supervisor, Tomas Persson, for helpful discussions
and for his guidance in preparing this paper. I am grateful to Viviane
Baladi for the reference \cite{pene2004Multiple} and to Nicholas
Fleming-V\'azquez for the reference
\cite[Theorem~7.41]{chernov2006Chaotic}.
I am also grateful to the referee for their comments, which helped
improve the presentation of the paper.


\end{document}